\documentclass[letterpaper,10pt]{article}

\usepackage[utf8]{inputenc}
\usepackage[T1]{fontenc}
\usepackage[top=1in,bottom=1in,left=1.5in,right=1.5in]{geometry}

\usepackage{amsmath,amssymb,amsfonts}
\usepackage{amsthm}
\usepackage{mathtools}

\usepackage[hidelinks]{hyperref}

\newcommand{\N}{\mathbb{N}}
\newcommand{\R}{\mathbb{R}}
\newcommand{\C}{\mathbb{C}}
\renewcommand{\Re}{\operatorname{Re}}

\DeclareMathOperator{\dom}{dom}
\DeclarePairedDelimiterX{\abs}[1]{\lvert}{\rvert}{#1}
\DeclarePairedDelimiterX{\norm}[1]{\lVert}{\rVert}{#1}
\DeclarePairedDelimiterX{\innprod}[2]{\langle}{\rangle}{#1,#2}

\newcommand{\spec}[1]{\sigma(#1)}

\renewcommand{\O}[1]{\mathcal{O}\bigl(#1\bigr)}
\newcommand{\Ol}[1]{\mathcal{O}\bigl(\lambda^{#1}\bigr)}

\theoremstyle{definition}
\newtheorem{theorem}{Theorem}[section]
\newtheorem{lemma}[theorem]{Lemma}
\newtheorem{proposition}[theorem]{Proposition}
\newtheorem{corollary}[theorem]{Corollary}

\renewenvironment{abstract}{%
	\small
	\centering
	\quote
	{\bfseries\abstractname.}%
}

\newenvironment{keywords}{%
	\small
	\centering
	\quote
	{\bfseries Keywords.}%
}

\title{\bfseries Spectral analysis of a coupled bending-torsion
	beam energy harvester: asymptotic results}

\author{Chris Vales\thanks{
	Mathematics \& Statistics, University of New Hampshire, Durham, NH, USA;
	Mathematics, Dartmouth College, Hanover, NH, USA
	(\texttt{chris.vales@dartmouth.edu}).}}

\date{}

\begin{document}

\maketitle

\begin{abstract}
This work is concerned with the spectral analysis of a
piezoelectric energy harvesting model based on a coupled bending-torsion
beam.
After building the problem's operator setting and showing that the
governing operator is nonselfadjoint with a purely discrete spectrum,
we derive an asymptotic approximation of its spectrum.
In doing so, we also prove that the addition of energy harvesting
can be viewed as a weak perturbation of the underlying beam dynamics,
in the sense that no piezoelectric parameters appear in the spectral
approximation's first two orders of magnitude.
We conclude by outlining future work based on numerical simulations.
\end{abstract}


\section{Introduction}\label{sec:intro}

A piezoelectric energy harvester is a device that utilizes the properties
of piezoelectric materials to convert mechanical strain energy into
electric energy.
The harvesting of mechanical vibration energy is currently considered for the
making of self powered microelectronic devices and remote sensors
\cite{Safaei2019,Liu2018,Cook2008},
with applications in areas such as health monitoring \cite{Inman2006}
and distributed sensor networks \cite{Roundy2004}.

In this work we consider a piezoelectric harvester extracting energy from the
vibration of a coupled bending-torsion beam.
Such coupled vibration naturally occurs when the beam's cross section has at
most one plane of symmetry, such as in the vibration of aircraft wings,
turbomachinery blading, and bridges \cite{Bal2004,Dokumaci1987}.
It can also occur by our breaking a beam's cross sectional symmetry on purpose,
which has been shown to increase the amount of harvested power compared to
purely bending vibration \cite{Abdelkefi2011,Abdelkefi2012}.

We consider a uniform cantilever beam with its left end fixed and its right
end free.
The beam's cross section is assumed to have only one plane of symmetry,
which causes the coupling of the bending and torsion motions.
A piezoelectric material layer is attached onto the top face of the beam
throughout its length.
Perfectly conducting electrodes at the top and bottom faces of the layer
connect it to a resistance load, thereby closing an electric circuit.
The layer is attached so that it harvests energy from the normal beam strain
due to bending, not the shear strain due to torsion.

The modeled quantities are
the beam's vertical centerline displacement $w(t,x)$,
the twist angle of its cross sections $\theta(t,x)$,
and the generated electric voltage $v(t)$,
where $t\geq 0$ denotes time and $x\in [0,L]$ space.
Using dot and prime to denote differentiation with respect to time and space
respectively, the complete electromechanical model reads
\begin{equation}\label{eq:harvester}
\begin{aligned}
	m\ddot{w}+S\ddot{\theta}+Ew''''&=0,\\
	J\ddot{\theta}+S\ddot{w}-G\theta''&=0,\\
	C_p\dot{v}+\frac{1}{R}v+C_D\dot{w}'(t,L)&=0,
\end{aligned}
\end{equation}
with boundary conditions
\begin{equation}\label{eq:harvester-BCs}
\begin{aligned}
	w(t,0)=w'(t,0)&=0,\\
	\theta(t,0)&=0,\\
	w'''(t,L)&=0,\\
	Ew''(t,L)+k_1\dot{w}'(t,L)+C_Iv&=0,\\
	G\theta'(t,L)+k_2\dot{\theta}(t,L)&=0
\end{aligned}
\end{equation}
and initial conditions
\begin{equation*}
\begin{aligned}
	w(0,x)&=w_0(x),\qquad& \dot{w}(0,x)&=w_1(x),\\
	\theta(0,x)&=\theta_0(x),\qquad& \dot{\theta}(0,x)&=\theta_1(x),\\
	v(0)&=v_0.
\end{aligned}
\end{equation*}
The first two equations of \eqref{eq:harvester} model the bending and
torsion vibration motions respectively, whereas the third equation
models the electric circuit.
In the mechanical equations,
$m$ denotes mass per unit length,
$J$ polar mass moment per unit volume,
$S$ the coupling constant with units of mass,
$E$ and $G$ the bending and torsion rigidity respectively.
In the electric equation,
$C_p$ is the piezoelectric layer's internal capacitance,
$R$ the load resistance,
and $C_D<0$ the direct (or forward) piezoelectric coupling
coefficient \cite{Erturk2008}.

In the boundary conditions \eqref{eq:harvester-BCs},
constant $C_I>0$ is the inverse (or backward) piezoelectric
coupling coefficient.
The terms involving positive constants $k_1$ and $k_2$
model vibration control mechanisms through strain-actuated damping
and are independent of the piezoelectric layer
\cite{Bal2000,Bal2004}.

In this work we derive an asymptotic approximation of the spectrum of
the operator governing the model's dynamics.
Moreover, we show that the addition of piezoelectric energy
harvesting can be considered a weak perturbation of the underlying
beam dynamics, in the sense that no piezoelectric parameters appear in
the first two orders of magnitude of the spectral approximation.
Our motivation for pursuing this line of research is to facilitate the
study of the geometric properties of the governing operator's set
of eigenvectors---particularly completeness, minimality, and the
Riesz basis property.

Asymptotic spectral investigations have been made in the past for various
beams and associated piezoelectric energy harvesters.
These include works on the Euler-Bernoulli beam \cite{Chen1987,Chen1990}
and associated energy harvester \cite{Shubov2016,Shubov2017};
the Rayleigh beam \cite{BRao1998,Shubov2014b};
the Timoshenko beam
\cite{Coleman1993,Geist1997,Geist1998,Geist2001,Shubov1999,Shubov2002}
and associated energy harvester \cite{Shubov2018};
the coupled bending-torsion beam
\cite{Dokumaci1987,Bishop1989,Bal2000,Bal2004,Shubov2004}.
Among other results, the Riesz basis property has been proven for
models based on the Euler-Bernoulli beam \cite{Shubov2017},
the Rayleigh beam \cite{BRao1998,Shubov2014b}
and the Timoshenko beam \cite{Shubov1999,Shubov2002}.

The present work is based on the analysis of the coupled bending-torsion
beam initiated in \cite{Bal2004,Shubov2004}
and is an effort to improve some aspects of it and refine its results.
More specifically, we extend the considered model by adding energy
harvesting and study its effect on the underlying beam dynamics.
We provide a detailed construction of the problem's operator setting,
identifying sufficient parameter conditions for its well-posedness.
Additionally, we refine the asymptotic analysis by including higher order
terms and by studying the limitations of the employed methodology.
In particular, we show that the derived asymptotic approximation is
valid for an infinite subset of the governing operator's spectrum,
but not necessarily for the whole spectrum as claimed in
\cite{Bal2004,Shubov2004}.
Finally, we propose numerical simulations that will be used to
supplement the derived asymptotic results.

In Section \ref{sec:setting} we build the operator setting of the problem
and describe the governing operator's spectrum.
The main results are presented in Section \ref{sec:main}
and proven in Section \ref{sec:proof},
followed by a discussion in Section \ref{sec:discussion}.
Appendix \ref{sec:omitted-proofs} contains proofs omitted in the
main text.
More detailed proofs for all results in this work can be found in
\cite{ValesPhD}.

\section{Operator setting}\label{sec:setting}

We define vector
\begin{equation*}
f:=\left(f_0(t,x),f_1(t,x),f_2(t,x),f_3(t,x),f_4(t)\right)
:=(w,\dot{w},\theta,\dot{\theta},v),
\end{equation*}
with smooth and complex valued component functions.
Assuming that $S<m$ and $S<J$, the model equations \eqref{eq:harvester}
are written as
\begin{equation}\label{eq:op-evolution}
\dot{f}=iAf,
\end{equation}
with governing operator
\begin{equation}\label{eq:op-generator}
A:=
\begin{bmatrix}
0 &-i &0 &0 &0\\
i\dfrac{EJ}{D}\dfrac{d^4}{dx^4} &0 &i\dfrac{GS}{D}\dfrac{d^2}{dx^2} &0 &0\\
0 &0 &0 &-i &0\\
-i\dfrac{ES}{D}\dfrac{d^4}{dx^4} &0 &-i\dfrac{Gm}{D}\dfrac{d^2}{dx^2} &0 &0\\
0 &-i\dfrac{C_D}{C_p}\delta_L' &0 &0 &i\dfrac{1}{C_pR}
\end{bmatrix},
\end{equation}
where $D:=mJ-S^2>0$, $\delta_L'g:=-g'(t,L)$, and time $t$ is from now on
treated as a parameter.
To define the appropriate domain for operator $A$, we consider space
\begin{equation*}
\widetilde{\mathcal{H}}:=\left\{f_i\in C^\infty([0,L]),
	f_4\in\C, i=0,\dotsc ,3:
	f_0(0)=0,\,f_0'(0)=0,\,f_2(0)=0\right\}.
\end{equation*}
and the following inner product, based on the system's mechanical and
electric energy,
\begin{equation}\label{eq:inner-prod}
\innprod{f}{g}:=\frac{1}{2}\int_0^L\left[Ef_0''\bar{g}_0''+mf_1\bar{g}_1
	+Gf_2'\bar{g}_2'+Jf_3\bar{g}_3+S\left(f_1\bar{g}_3
	+f_3\bar{g}_1\right)\right]dx+\frac{1}{2}C_pf_4\bar{g}_4,
\end{equation}
with $f$, $g\in\widetilde{\mathcal{H}}$
and bars denoting complex conjugation.

\begin{lemma}\label{lem:inner-prod}
The functional
$\innprod{\cdot}{\cdot}:\widetilde{\mathcal{H}}
\times\widetilde{\mathcal{H}}\rightarrow\C$
of \eqref{eq:inner-prod} forms an
inner product in $\widetilde{\mathcal{H}}$.
\end{lemma}
\begin{proof}
See Appendix \ref{sec:omitted-proofs}.
\end{proof}

\begin{lemma}\label{lem:norm-equiv}
Let
$\mathcal{H}_1:=H^2((0,L))\times L^2([0,L])\times H^1((0,L))
\times L^2([0,L])\times\C$
with its usual product norm
\begin{equation}\label{eq:norm1}
	\norm{f}_1^2:=\int_0^L\biggl(\sum_{j=0}^2\abs{f_0^{(j)}}^2
		+\abs{f_1}^2+\sum_{k=0}^1\abs{f_2^{(k)}}^2+\abs{f_3}^2\biggr)dx
		+\abs{f_4}^2.
\end{equation}
If $S<m$ and $S<J$, then norm \eqref{eq:norm1}
and the norm induced by \eqref{eq:inner-prod}
are equivalent in $\widetilde{\mathcal{H}}$.
\end{lemma}
\begin{proof}
See Appendix \ref{sec:omitted-proofs}.
\end{proof}

The well-posedness of the operator setting rests on the condition
$S<\min (m, J)$,
where $S$ is the coupling constant of the bending and torsion motions.
Therefore, the present work and derived results are valid for
\emph{weakly} coupled bending-torsion beams.

Using the norm equivalence result, we define the problem's state space
as the closure of $\widetilde{\mathcal{H}}$ in the topology generated
by the norm induced by \eqref{eq:inner-prod}, which yields the
Hilbert space
\begin{multline}\label{eq:hilbert-space}
\mathcal{H} = \Bigl\{f_0\in H^2((0,L)),\, f_1,f_3\in L^2([0,L]),\,
	f_2\in H^1((0,L)),\, f_4\in\C :\Bigr.\\
	\Bigl. f_0(0)=0,\, f_0'(0)=0,\, f_2(0)=0\Bigr\}.
\end{multline}
The domain of operator $A$ is then defined as
\begin{equation}\label{eq:op-domain}
\begin{split}
	\dom A:=\Bigl\{f\in\mathcal{H}:\,&f_0\in H^4((0,L)),\,f_1,f_2\in
		H^2((0,L)),\,f_3\in H^1((0,L)),\,f_4\in\C,\Bigr.\\
	&\Bigl.f_1(0)=0,\,f_1'(0)=0,\,f_3(0)=0,\,f_0'''(L)=0,\Bigr.\\
	&\Bigl.Ef_0''(L)+k_1f_1'(L)+C_If_4=0,\,
		Gf_2'(L)+k_2f_3(L)=0\Bigr\}.
\end{split}
\end{equation}

\begin{lemma}\label{lem:domain-dense}
The domain of operator $A$ is dense in $\mathcal{H}$.
\end{lemma}
\begin{proof}
See Appendix \ref{sec:omitted-proofs}.
\end{proof}

Next, we show that $A$ is nonselfadjoint and has a compact inverse,
properties we will use to describe its spectrum.

\begin{proposition}
Operator $A$ is unbounded, closed and nonselfadjoint.
\end{proposition}
\begin{proof}
To verify the relation
\begin{equation*}
	\innprod{Af}{g}=\innprod{f}{A^*g}\quad\forall
		f\in\dom A ,\, g\in\dom {A^*},
\end{equation*}
let operator $A^*$ have the following matrix
\begin{equation}\label{eq:op-adjoint}
A^*=\begin{bmatrix}
	0 &-i &0 &0 &0\\
	i\dfrac{EJ}{D}\dfrac{d^4}{dx^4} &0 &i\dfrac{GS}{D}\dfrac{d^2}{dx^2}
		&0 &0\\
	0 &0 &0 &-i &0\\
	-i\dfrac{ES}{D}\dfrac{d^4}{dx^4} &0 &-i\dfrac{Gm}{D}\dfrac{d^2}{dx^2}
		&0 &0\\
	0 &i\dfrac{C_I}{C_p}\delta_L' &0 &0 &-i\dfrac{1}{C_pR}
	\end{bmatrix}
\end{equation}
and domain
\begin{equation}\label{eq:adjoint-domain}
\begin{split}
	\dom A^*:=\Bigl\{f\in\mathcal{H}:\,&f_0\in H^4((0,L)),\,
		f_1,f_2\in H^2((0,L)),\,
		f_3\in H^1((0,L)),\,f_4\in\C,\Bigr.\\
	&\Bigl.f_1(0)=0,\,f_1'(0)=0,\,f_3(0)=0,\,f_0'''(L)=0,\Bigr.\\
	&\Bigl.Ef_0''(L)-k_1f_1'(L)-C_Df_4=0,\,
		Gf_2'(L)-k_2f_3(L)=0\Bigr\}.
\end{split}
\end{equation}
Performing the required integrations by parts and enforcing the boundary
conditions for $f\in\dom A$ and $g\in\dom A^*$ yields
\begin{equation*}
	\innprod{Af}{g}-\innprod{f}{A^*g}=0,
\end{equation*}
which means that $A^*\neq A$.

Given the minor differences between $\dom A$ and $\dom A^*$,
it follows that $\dom A^*$ is also dense in $\mathcal{H}$,
so $A^{**}$ is well defined.
Repeating the above calculation shows that $A^{**}=A$.
As the adjoint of densely defined $A^*$, $A^{**}$ is closed
\cite{BirmanSolomiak}.
\end{proof}

The above proof demonstrates that $A$ is nonselfadjoint because of the
contributions from the electric circuit equation and the right-end boundary
conditions involving the piezoelectric and control parameters.
If only the mechanical system with $k_1=k_2=0$ is considered, then the
corresponding operator is selfadjoint.

\begin{proposition}
Operator $A$ is invertible and its inverse is compact.
\end{proposition}
\begin{proof}
Let $g\in\mathcal{H}$; we show that there exists unique
$f\in\dom A$ such that $Af=g$.
Expanded into its components, $Af=g$ can be written as
\begin{align*}
	f_1&=ig_0,\\
	Ef_0''''&=-i(mg_1+Sg_3)\\
	f_3&=ig_2,\\
	Gf_2''&=i(Sg_1+Jg_3)\\
	f_4&=-iC_pR\bigl[g_4+\frac{C_D}{C_p}g_0'(L)\bigr].
\end{align*}
The second and fourth equations yield respectively
\begin{multline*}
	f_0(x)=-i\frac{1}{E}\int_0^x\int_0^{x_4}\int_{x_3}^L\int_{x_2}^L\left[
		mg_1(x_1)+Sg_3(x_1)\right]dx_1dx_2dx_3dx_4\\
		+i\frac{C_IC_DR-k_1}{2E}g_0'(L)x^2+i\frac{C_IC_pR}{2E}g_4x^2
\end{multline*}
and
\begin{equation*}
	f_2(x)=-i\frac{1}{G}\int_0^x\int_{x_2}^L\left[Sg_1(x_1)
		+Jg_3(x_1)\right]dx_1dx_2-i\frac{k_2}{G}g_2(L)x
\end{equation*}
after the appropriate boundary conditions of $\dom A$ are enforced.
The above equations uniquely define $f$ in terms of $g$.
Further inspection shows that component functions $f_i$ belong to the
appropriate function spaces and satisfy all boundary conditions of $\dom A$;
namely, $f\in\dom A$.

To prove compactness, let
\begin{align*}
	\mathcal{H}_1&:=H^2((0,L))\times L^2([0,L])\times H^1((0,L))
		\times L^2([0,L])\times\C ,\\
	\mathcal{H}_2&:=H^4((0,L))\times H^2((0,L))\times H^2((0,L))
		\times H^1((0,L))\times\C
\end{align*}
with their usual product norms, so that
$\mathcal{H}\subset\mathcal{H}_1$ and
$\dom A\subset\mathcal{H}_2$.
The above calculations imply that map
$A^{-1}:\mathcal{H}\subset\mathcal{H}_1\rightarrow\mathcal{H}_2$
is bounded.
Since the embedding
$\mathcal{H}_2\hookrightarrow\mathcal{H}_1$
is compact \cite{AdamsSobolev},
operator
$A^{-1}:\mathcal{H}\subset\mathcal{H}_1\rightarrow\mathcal{H}_1$ is
compact too.
From the equivalence of the norm induced topologies of $\mathcal{H}$
and $\mathcal{H}_1$, it follows that
$A^{-1}:\mathcal{H}\rightarrow\mathcal{H}$
is compact.
\end{proof}

\begin{corollary}\label{cor:spectrum-properties}
Operator $A$ has a purely discrete spectrum, which has $\infty$
as its limit point and is symmetric about the imaginary axis.
\end{corollary}
\begin{proof}
Being compact, operator $A^{-1}$ has a purely discrete spectrum with
0 as its limit point, leading to the first result.
Namely, $A$'s spectrum consists entirely of isolated eigenvalues
with finite algebraic multiplicities.
Next, direct calculation shows that pair $(\lambda,f)$
is a solution of the spectral problem $Af=\lambda f$
if and only if $(-\bar{\lambda},\bar{f})$ is one, which yields
the second result.
\end{proof}

\begin{lemma}\label{lem:positive-imag}
If $C_I=-C_D$, then all eigenvalues of $A$ have positive imaginary part. 
\end{lemma}
\begin{proof}
See Appendix \ref{sec:omitted-proofs}.
\end{proof}

\section{Main results}\label{sec:main}

We now present the main results of this work.
Their proof is given in the next section, followed by a discussion.
All results are subject to the standing assumptions
$S<m$, $S<J$ and $C_I=-C_D$.
In addition, $\spec{A}$ is used to denote the spectrum of operator $A$,
while the symbols of complex multivalued functions are used to denote 
their principal branch.

\begin{theorem}[Weak perturbation]\label{thm:weak-perturbation}
The model's piezoelectric parameters do not appear in the first two
orders of magnitude of the asymptotic approximation of
$\lambda\in\sigma(A)$ as $\lambda\rightarrow\infty$.
The third order terms are the first to be modified by the addition
of piezoelectric energy harvesting.
\end{theorem}

\begin{theorem}[Eigenvalue asymptotics]\label{thm:eigval-asymptotics}
The leading order term of the asymptotic approximation of
$\lambda\in\sigma(A)$ as $\lambda\rightarrow\infty$
consists of the two disjoint subsets
\begin{align}\label{eq:thm-branch12-leading}
	&\tilde{\lambda}_{1,n}=\frac{\pi\sqrt{G}}{L\sqrt{J}}n-i\frac{\sqrt{G}}
		{2L\sqrt{J}}\ln\frac{k_2-\sqrt{GJ}}{k_2+\sqrt{GJ}},
		\quad n\in\N,\quad n\rightarrow\infty,\\
	&\tilde{\lambda}_{2,n}=(n-1/4)^2\frac{\pi^2\sqrt{EJ}}{L^2\sqrt{mJ-S^2}},
		\quad n\in\N, \quad n\rightarrow\infty,
\end{align}
referred to as the unperturbed branches 1 and 2 respectively,
subject to condition
\begin{equation*}
	k_2>\sqrt{GJ}.
\end{equation*}
Furthermore, there exists an infinite subset
$\sigma^*(A)\subset\sigma(A)$
such that the second order asymptotic approximation of $\sigma^*(A)$
consists of two disjoint subsets referred to as the perturbed
branches 1 and 2.

The asymptotic approximation of perturbed branch 1 reads
\begin{equation}\label{eq:thm-branch1-order2}
	\lambda_{1,n}=\tilde{\lambda}_{1,n}\bigl[
		1+w_{1,n}+\O{n^{-2}}\bigr],\quad n\in\N^*,
		\quad n\rightarrow\infty,
\end{equation}
with infinite subset $\N^*\subset\N$,
\begin{equation}\label{eq:thm-branch1-w1}
	w_{1,n}=-i\frac{\sqrt{G}}{2L\sqrt{J}}\tilde{\lambda}_{1,n}^{-1}\ln K_{1,n}
		=\O{n^{-3/2}},
\end{equation}
and $K_{1,n}$ given by \eqref{eq:branch1-K1}.
For each $\lambda_{1,n}$ the branch also contains $-\bar{\lambda}_{1,n}$.

The asymptotic approximation of perturbed branch 2 reads
\begin{equation}\label{eq:thm-branch2-order2}
	\lambda_{2,n}=\tilde{\lambda}_{2,n}\left[1+w_{2,n}+\O{n^{-3}}\right],
		\quad n\in\N,\quad n\rightarrow\infty,
\end{equation}
with $w_{2,n}=\O{n^{-2}}$ given by \eqref{eq:branch2-w2} in implicit form.
For each $\lambda_{2,n}$ the branch also contains $-\bar{\lambda}_{2,n}$.
\end{theorem}

\section{Proof of main results}\label{sec:proof}

In this section we prove Theorems \ref{thm:weak-perturbation} and
\ref{thm:eigval-asymptotics}.
We denote by $\lambda\in\C\setminus\{ 0\}$ an eigenvalue and by
$f\in\dom A$  an eigenvector of operator $A$.
Written out to its components, the spectral equation
\begin{equation}\label{eq:spectral-A}
Af=\lambda f
\end{equation}
reads
\begin{align*}
	f_1&=i\lambda f_0\\
	i\frac{EJ}{D}f_0''''+i\frac{GS}{D}f_2''&=\lambda f_1\\
	f_3&=i\lambda f_2\\
	-i\frac{ES}{D}f_0''''-i\frac{Gm}{D}f_2''&=\lambda f_3\\
	f_4&=-\frac{C_D/C_p}{\lambda -i\frac{1}{C_pR}}\lambda f_0'(L)\,,
\end{align*}
assuming that $\lambda\neq i(C_pR)^{-1}$.
Combining the first and second equations with the third and fourth yields
\begin{equation*}
EGf_0^{(6)}+EJ\lambda ^2f_0''''-Gm\lambda ^2f_0''-D\lambda ^4f_0=0,
\end{equation*}
which is the spectral equation written in terms of component function
$f_0$ only.
We now rewrite the boundary conditions encoded in $\dom A$ in a form that
involves only $f_0$,
\begin{align*}
	f_0(0)=f_0'(0)=f_0''''(0)&=0,\\
	f_0'''(L)&=0,\\
	Ef_0''(L)+ ik_1\lambda f_0'(L)-\frac{C_IC_D}{C_p}
		\frac{\lambda}{\lambda-i\frac{1}{C_pR}}f_0'(L)&=0\\
	EGf_0^{(5)}(L)+ik_2E\lambda f_0''''(L)-Gm\lambda^2f_0'(L)-
		ik_2m\lambda^3f_0(L)&=0.
\end{align*}

Based on the above results, we define operator pencil $A_0(\cdot)$ by
\begin{equation}\label{eq:A0-operator}
	A_0(\lambda)f_0:=EGf_0^{(6)}+EJ\lambda ^2f_0''''-Gm\lambda ^2f_0''
		-D\lambda ^4f_0
\end{equation}
and
\begin{equation}\label{eq:A0-domain}
\begin{split}
	\dom A_0(\lambda):=\,&\Bigl\{f_0\in H^6((0,L)):f_0(0)=0,\,f_0'(0)=0,\,
		f_0''''(0)=0,\,f_0'''(L)=0,\Bigr.\\
	&\quad Ef_0''(L)+\Bigl(ik_1-\frac{C_IC_D}{C_p}\frac{1}{\lambda
		-i\frac{1}{C_pR}}\Bigr)\lambda f_0'(L)=0,\\
	&\quad\Bigl.EGf_0^{(5)}(L)+ik_2E\lambda f_0''''(L)-Gm\lambda^2f_0'(L)
	-ik_2m\lambda^3f_0(L)=0\Bigr\}.
\end{split}
\end{equation}
The spectral problem for A can now be rewritten as the equivalent problem
\begin{equation}\label{eq:spectral-A0}
	A_0(\lambda)f_0=0,
\end{equation}
where a nontrivial pair $(\lambda,f_0)$ is a solution to
\eqref{eq:spectral-A0}
if and only if $(\lambda, f)$ is a solution to
\eqref{eq:spectral-A}.

\subsection{Characteristic roots}

We look for solutions to problem \eqref{eq:spectral-A0} of the form
$f_0(x)\propto e^{\zeta x}$, $\zeta\in\C$.
Substituting that into the above equation and making the changes of variable
$y:=\zeta ^2$ and $z:=y+\alpha/3\lambda ^2$
with
\begin{equation}\label{eq:abc-const}
	\alpha:=\frac{J}{G},\qquad\beta:=\frac{m}{E},
		\qquad\gamma:=\frac{D}{EG},
\end{equation}
yields the depressed cubic equation
\begin{equation*}
	z^3+pz+q=0,
\end{equation*}
where
\begin{equation*}
	p:=-\Bigl(\frac{\alpha^2}{3}\lambda^4+\beta\lambda^2\Bigr),\qquad
	q:=\frac{2}{27}\alpha^3\lambda^6+\Bigl(\frac{\alpha\beta}{3}
		-\gamma\Bigr)\lambda^4.
\end{equation*}
Using Cardano's formulas the equation's three roots can be found as
\begin{equation}\label{eq:roots123}
	z_j=\hat{z}_j-\tilde{z}_j,\quad j=1,2,3,
\end{equation}
with
\begin{equation}\label{eq:hattilde1}
	\hat{z}_1=\sqrt[3]{-\frac{q}{2}+\sqrt{\left(\frac{q}{2}\right)^2
		+\left(\frac{p}{3}\right)^3}},\qquad
	\tilde{z}_1=\sqrt[3]{\frac{q}{2}+\sqrt{\left(\frac{q}{2}\right)^2
		+\left(\frac{p}{3}\right)^3}}
\end{equation}
and
\begin{equation}\label{eq:hattilde23}
	\hat{z}_j=\exp{\Bigl[i\frac{2\pi}{3}(j-1)\Bigr]}\hat{z}_1,\qquad
	\tilde{z}_j=\exp{\Bigl[-i\frac{2\pi}{3}(j-1)\Bigr]}\tilde{z}_1,
		\quad j=2,3.
\end{equation}

Employing \eqref{eq:hattilde1} and expanding the involved roots as
$\lambda\rightarrow\infty$
using the generalized binomial expansion theorem leads to
\begin{equation*}
	z_1=-\frac{2J}{3G}\lambda ^2-\frac{GS^2}{EJ^2}+\Ol{-2}.
\end{equation*}
Reversing the changes of variable yields
\begin{equation*}
	\zeta_{1,2}=\pm\,\sqrt{z_1-\frac{\alpha}{3}\lambda^2}
		=\pm\,i\Bigl(\sqrt{\frac{J}{G}}\lambda
		+\frac{G^{3/2}S^2}{2EJ^{5/2}}\lambda^{-1}
		\Bigr)+\Ol{-3}.
\end{equation*}
Unless stated otherwise, all asymptotic results are understood in
the limit $\lambda\rightarrow\infty$.
Similarly, using \eqref{eq:hattilde23} we calculate
\begin{equation*}
	\zeta_{3,4}=\pm\,\sqrt{z_2-\frac{\alpha}{3}\lambda^2}
		=\pm\,i\Bigl[\Bigl(\frac{D}{EJ}\Bigr)^{1/4}\lambda^{1/2}
		-\frac{GS^2}{4E^{3/4}J^{7/4}D^{1/4}}\lambda^{-1/2}\Bigr]
		+\Ol{-3/2},
\end{equation*}
and
\begin{equation*}
	\zeta_{5,6}=\pm\,\sqrt{z_3-\frac{\alpha}{3}\lambda^2}
		=\pm\,\Bigl[\Bigl(\frac{D}{EJ}\Bigr)^{1/4}\lambda^{1/2}
		+\frac{GS^2}{4E^{3/4}J^{7/4}D^{1/4}}\lambda^{-1/2}\Bigr]
		+\Ol{-3/2}.
\end{equation*}
Finally, using constants
\begin{align}\label{eq:a12-const}
	a_1&:=\Bigl(\frac{J}{G}\Bigr)^{1/2},&
	a_2&:=\frac{G^{3/2}S^2}{2EJ^{5/2}},\\
	\label{eq:a34-const}
	a_3&:=\Bigl(\frac{D}{EJ}\Bigr)^{1/4},&
	a_4&:=\frac{GS^2}{4E^{3/4}J^{7/4}D^{1/4}},
\end{align}
the characteristic roots $\zeta_j$ of \eqref{eq:spectral-A0}
are written as
\begin{align}\label{eq:zeta12-char}
	\zeta_{1,2}&=\pm ia_1\lambda\Bigl[1+\frac{a_2}{a_1}\lambda^{-2}
		+\Ol{-4}\Bigr],\\
	\label{eq:zeta34-char}
	\zeta_{3,4}&=\pm ia_3\lambda^{1/2}\Bigl[1-\frac{a_4}{a_3}\lambda^{-1}
		+\Ol{-2}\Bigr],\\
	\label{eq:zeta56-char}
	\zeta_{5,6}&=\pm a_3\lambda^{1/2}\Bigl[1+\frac{a_4}{a_3}\lambda^{-1}
		+\Ol{-2}\Bigr].
\end{align}

\subsection{Reduced spectral equation}

We express the solution to \eqref{eq:spectral-A0} as a linear combination
of the terms $e^{\zeta_j}$
\begin{equation*}
	f_0(x)=\sum_{j=1,3,5}b_je^{\zeta_jx}+\sum_{k=1,3,5}c_ke^{\zeta_{k+1}x}
		=\sum_{j=1,3,5}b_je^{\zeta_jx}+\sum_{k=1,3,5}c_ke^{-\zeta_kx},
\end{equation*}
with $\C^3$ constants $b:=(b_1,b_3,b_5)^T$,
$c:=(c_1,c_3,c_5)^T$.
Next, we enforce the six boundary conditions encoded in
\eqref{eq:A0-domain}
using the \emph{reflection matrices method}
\cite{Chen1990,Coleman1993,Shubov2004}.

The three left-end boundary conditions produce a system of three equations
for the six unknown constants, which can be written as
\begin{equation}\label{eq:left-reflection}
	b=R_1c,
\end{equation}
with $3\times 3$ matrix $R_1$ termed the \emph{left reflection matrix}.
Similarly, the three right-end boundary conditions lead to system
\begin{equation}\label{eq:right-reflection}
	b=R_2c,
\end{equation}
where the $3\times 3$ matrix $R_2$ is called the
\emph{right reflection matrix}.
Assuming $R_1$ is invertible, we combine the two equations to write 
\begin{align*}
	\begin{pmatrix}b\\c\end{pmatrix}&=
		\begin{bmatrix}0&R_2\\R_1^{-1}&0\end{bmatrix}
		\begin{pmatrix}b\\c\end{pmatrix}\\
	\left(I-\begin{bmatrix}0&R_2\\R_1^{-1}&0\end{bmatrix}\right)
	\begin{pmatrix}b\\c\end{pmatrix}&=0,
\end{align*}
which admits a nontrivial solution if and only if
\begin{align}
	\det\left(I-\begin{bmatrix}0&R_2\\R_1^{-1}&0\end{bmatrix}\right)&=0
		\nonumber\\
	\det (I-R_1^{-1}R_2)&=0\nonumber\\
	\label{eq:spectral-reduced}
	\det (R_1-R_2)&=0.
\end{align}
Through these manipulations, enforcing the boundary conditions has been
reduced to solving the reduced spectral equation
\eqref{eq:spectral-reduced},
which involves the determinant of a $3\times 3$ matrix.
To derive \eqref{eq:spectral-reduced}
we calculate the asymptotic approximation of the two reflection matrices.

\subsection{Left reflection matrix}

The three left-end boundary conditions yield the system of equations
\begin{align*}
	\begin{bmatrix}1&1&1\\ \zeta_1&\zeta_3&\zeta_5\\
		\zeta_1^4&\zeta_3^4&\zeta_5^4\end{bmatrix}
		\begin{pmatrix}b_1\\b_3\\b_5\end{pmatrix}&=
		\begin{bmatrix}-1&-1&-1\\ \zeta_1&\zeta_3&\zeta_5\\
		-\zeta_1^4&-\zeta_3^4&-\zeta_5^4\end{bmatrix}
		\begin{pmatrix}c_1\\c_3\\c_5\end{pmatrix}\\
	A_1b&=B_1c\\
	b&=A_1^{-1}B_1c\\
	b&=R_1c.
\end{align*}
Using matrix
\begin{equation*}
	B_2:=\begin{bmatrix}0&0&0\\ \zeta_1&\zeta_3&\zeta_5\\0&0&0
		\end{bmatrix}
\end{equation*}
we write
\begin{equation}\label{eq:R1-calculation}
	R_1=A_1^{-1}B_1=A_1^{-1}\left(-A_1+2B_2\right)=-I+2A_1^{-1}B_2.
\end{equation}
Since the second row of $B_2$ is its only nonzero row, we only need
calculate the second column of $A_1^{-1}$ to determine $R_1$.

To do that we employ the method of cofactors, which requires that we
first calculate the determinant of $A_1$.
Expanding about the matrix's third row and employing
\eqref{eq:zeta12-char}--\eqref{eq:zeta56-char}
we find
\begin{equation*}
	\frac{1}{\det A_1}=\frac{1}{a_1^4a_3(1-i)}\lambda^{-9/2}
		\left[1-ia_4\lambda^{-1}+\Ol{-2}\right].
\end{equation*}
Denoting by $A_1^{-1}(j,k)$ the $(j,k)$ entry of $A_1^{-1}$
and by $C_{jk}$ its $(j,k)$ cofactor,
we write 
\begin{align*}
	A_1^{-1}(1,2)&=\frac{C_{21}}{\det A_1}=\Ol{-7/2},\\
	A_1^{-1}(2,2)&=\frac{C_{22}}{\det A_1}=-\frac{1}{a_3(1-i)}\lambda^{-1/2}
		\left[1-ia_4\lambda^{-1}+\Ol{-2}\right],\\
	A_1^{-1}(3,2)&=\frac{C_{23}}{\det A_1}=\frac{1}{a_3(1-i)}\lambda^{-1/2}
		\left[1-ia_4\lambda^{-1}+\Ol{-2}\right]
\end{align*}
and, after using \eqref{eq:R1-calculation}, 
\begin{align*}
	R_1(1,1)&=-1+\Ol{-5/2},\\
	R_1(1,2)&=\Ol{-3},\\
	R_1(1,3)&=\Ol{-3},\\
	R_1(2,1)&=\frac{2a_1}{a_3(1+i)}\lambda^{1/2}\left[1-ia_4\lambda^{-1}
		+\Ol{-2}\right],\\
	R_1(2,2)&=-i-\frac{2a_4}{1+i}\frac{1+ia_3}{a_3}\lambda^{-1}+\Ol{-2},\\
	R_1(2,3)&=-\frac{2}{1-i}\Bigl[1+a_4\frac{1-ia_3}{a_3}\lambda^{-1}
		+\Ol{-2}\Bigr],\\
	R_1(3,1)&=-\frac{2a_1}{a_3(1+i)}\lambda^{1/2}\left[1-ia_4\lambda^{-1}
		+\Ol{-2}\right],\\
	R_1(3,2)&=-\frac{2}{1+i}\Bigl[1-a_4\frac{1+ia_3}{a_3}\lambda^{-1}
		+\Ol{-2}\Bigr],\\
	R_1(3,3)&=i+\frac{2a_4}{1-i}\frac{1-ia_3}{a_3}\lambda^{-1}+\Ol{-2}.
\end{align*}
Direct calculation shows that
\begin{equation*}
	\det R_1=1+\Ol{-5/2}\neq 0,
\end{equation*}
which means that $R_1$ is indeed invertible.

\subsection{Right reflection matrix}

We define
\begin{align}
	\label{eq:e135-func}
	e_j&:=e^{\zeta_jL},\quad j=1,3,5,\\
	\hat{c}&:=ik_1-\frac{C_IC_D}{C_p}\Bigl(\lambda
		-i\frac{1}{C_pR}\Bigr)^{-1}\nonumber
\end{align}
and the $3\times 3$ matrices
\begin{align*}
	A_3&:=\begin{bmatrix}\zeta_j^3\\\,\,E\zeta_j^2+\hat{c}\lambda\zeta_j\\
		EG\zeta_j^5+iEk_2\lambda\zeta_j^4-Gm\lambda^2\zeta_j-imk_2\lambda^3
		\end{bmatrix}
		\quad\text{with $j=1,3,5$},\\
	B_3&:=\begin{bmatrix}\zeta_j^3\\ \hat{c}\lambda\zeta_j-E\zeta_j^2\\
		EG\zeta_j^5-iEk_2\lambda\zeta_j^4-Gm\lambda^2\zeta_j+imk_2\lambda^3
		\end{bmatrix}
		\quad\text{with $j=1,3,5$}
\end{align*}
and
\begin{equation*}
	\bar{E}:=\text{diag}(e_1,e_3,e_5),
\end{equation*}
where values $j=1,3,5$ correspond to columns one, two and three
respectively.

Now, the three right-end boundary conditions generate the system
\begin{align*}
	A_3\bar{E}b&=B_3\bar{E}^{-1}c\\
	b&=\bar{E}^{-1}A_3^{-1}B_3\bar{E}^{-1}c\\
	b&=R_2c.
\end{align*}
Using $3\times 3$ matrix
\begin{equation*}
	B_4:=\begin{bmatrix}0\\E\zeta_j^2\\iEk_2\lambda\zeta_j^4
		-imk_2\lambda^3\end{bmatrix}
	\quad\text{with $j=1,3,5$},
\end{equation*}
where again values $j=1,3,5$ correspond to columns one, two and three,
we simplify the calculation of $R_2$ to
\begin{equation}\label{eq:R2-calculation}
	R_2=\bar{E}^{-1}A_3^{-1}B_3\bar{E}^{-1}
	=\bar{E}^{-1}A_3^{-1}\left(A_3-2B_4\right)\bar{E}^{-1}
	=\bar{E}^{-1}\left(I-2A_3^{-1}B_4\right)\bar{E}^{-1}.
\end{equation}
Since the first row of $B_4$ is zero, only columns two and three of
$A_3^{-1}$ are needed for determining $R_2$.

As before, to determine the desired entries of $A_3^{-1}$ we begin with
calculating $A_3$'s determinant,
\begin{equation*}
	\frac{1}{\det A_3}=\frac{1}{i2Ek_1a_1^4a_3^4(Ga_1+k_2)}\lambda^{-8}
		\left[1-d_1\lambda^{-1/2}+d_2\lambda^{-1}+\Ol{-3/2}\right],
\end{equation*}
where
\begin{align*}
	d_1&:=\frac{(1-i)Ea_3}{2k_1}+\frac{(i-1)k_2\left(Ea_3^4-m\right)}
		{2Ea_1a_3^3(Ga_1+k_2)},\\
	\hat{d}_2&:=i\frac{C_IC_D}{k_1C_p}+i\frac{k_2\left(Ea_3^4-m\right)}
		{k_1a_1a_3^2(Ga_1+k_2)},\\
	d_2&:=d_1^2-\hat{d}_2.
\end{align*}
Next, we calculate the required entries of matrices $A_3^{-1}$ and $B_4$
and combine them to derive the entries of $I-2A_3^{-1}B_4$
which appears in \eqref{eq:R2-calculation}.
Denoting its entries simply by $(j,k)$ we have
\begin{align*}
	(1,1)&=1-r_{11}\left[1+\hat{r}_{11}\lambda^{-1/2}
		+\tilde{r}_{11}\lambda^{-1}+\Ol{-3/2}\right],\\
	(1,2)&=r_{12}\lambda^{-2}\left[1-\hat{r}_{12}\lambda^{-1/2}
		+\Ol{-1}\right],\\
	(1,3)&=r_{12}\lambda^{-2}\left[1+\hat{r}_{13}\lambda^{-1/2}
		+\Ol{-1}\right],\\
	(2,1)&=r_{21}\lambda^{3/2}\left[1-\hat{r}_{12}\lambda^{-1/2}
		+\Ol{-1}\right],\\
	(2,2)&=1-r_{22}\lambda^{-1/2}\left[1+(\hat{r}_{22}-d_1)\lambda^{-1/2}
		+\Ol{-1}\right],\\
	(2,3)&=r_{23}\lambda^{-1/2}\left[1-d_1\lambda^{-1/2}+\Ol{-1}\right],\\
	(3,1)&=-ir_{21}\lambda^{3/2}\left[1+\hat{r}_{13}\lambda^{-1/2}
		+\Ol{-1}\right],\\
	(3,2)&=-ir_{23}\lambda^{-1/2}\left[1-d_1\lambda^{-1/2}
		+\Ol{-1}\right],\\
	(3,3)&=1+ir_{22}\lambda^{-1/2}\left[1+(i\hat{r}_{22}-d_1)\lambda^{-1/2}
		+\Ol{-1}\right],
\end{align*}
with constants
\begin{align*}
	r_{11}&:=\frac{2k_2}{Ga_1+k_2},&
	\hat{r}_{11}&:=\frac{(1-i)Ea_3}{2k_1}-d_1,\\
	\tilde{r}_{11}&:=i\frac{C_IC_D}{k_1C_p}-\frac{(1-i)Ea_3d_1}{2k_1}+d_2,&
	r_{12}&:=-\frac{2k_2\left(Ea_3^4-m\right)}{Ea_1^4(Ga_1+k_2)},\\
	\hat{r}_{12}&:=i\frac{Ea_3}{k_1}+d_1,&
	\hat{r}_{13}&:=\frac{Ea_3}{k_1}-d_1\\
	r_{21}&:=\frac{k_2a_1^3}{a_3^3(Ga_1+k_2)}
\end{align*}
and
\begin{align*}
	r_{22}&:=\frac{E^2a_1a_3^4(Ga_1+k_2)-k_1k_2\left(Ea_3^4-m\right)}
		{Ek_1a_1a_3^3(Ga_1+k_2)},\\
	\hat{r}_{22}&:=i\frac{2Ek_2a_3\left(Ea_3^4-m\right)}
		{E^2a_1a_3^4(Ga_1+k_2)-k_1k_2\left(Ea_3^4-m\right)},\\
	r_{23}&:=\frac{E^2a_1a_3^4(Ga_1+k_2)+k_1k_2\left(Ea_3^4-m\right)}
		{Ek_1a_1a_3^3(Ga_1+k_2)}.
\end{align*}

Multiplying $I-2A_3^{-1}B_4$ from left and right by diagonal
matrix $\bar{E}^{-1}$ yields
\begin{equation*}
	R_2=\begin{bmatrix}
		e_1^{-2}(1,1)&e_1^{-1}e_3^{-1}(1,2)&e_1^{-1}e_5^{-1}(1,3)\\
		e_1^{-1}e_3^{-1}(2,1)&e_3^{-2}(2,2)&e_3^{-1}e_5^{-1}(2,3)\\
		e_1^{-1}e_5^{-1}(3,1)&e_3^{-1}e_5^{-1}(3,2)&e_5^{-2}(3,3)
		\end{bmatrix},
\end{equation*}
where $(j,k)$ denotes the entries of $I-2A_3^{-1}B_4$ calculated above.

\subsection{Solving the spectral equation}

We are now in a position to derive and solve a modified version of the
reduced spectral equation \eqref{eq:spectral-reduced}.
We start by considering the behavior of functions $e_j(\lambda)$,
$j=1,3,5$, $\lambda\in\C\setminus\{0\}$, defined in \eqref{eq:e135-func}.
Let
\begin{equation}
	\lambda=:x+iy,\qquad \lambda^{1/2}=:u+iv,
\end{equation}
with $x\in\R$ and $y$, $u$, $v\geq 0$.
Since the set of eigenvalues is symmetric about the imaginary axis,
we need only consider the case where $\lambda$ is in the complex plane's
first quadrant; namely, $x$, $y\geq 0$.
Consequently, $\lambda^{1/2}$ is also in the first quadrant, and
particularly in the triangular domain below straight line $u=v$,
with $u\rightarrow\infty$ as $\lambda\rightarrow\infty$.

Using \eqref{eq:zeta12-char}-\eqref{eq:zeta56-char},
we expand functions $e_j$ to leading order
\begin{align*}
	e_1(\lambda)&=e^{ia_1L\lambda}\left[1+\Ol{-1}\right]=
		e^{-a_1Ly}e^{ia_1Lx}\left[1+\Ol{-1}\right],\\
	e_3(\lambda)&=e^{ia_3L\lambda^{1/2}}\Bigl[1+\Ol{-1/2}\Bigr]=
		e^{-a_3Lv}e^{ia_3Lu}\Bigl[1+\Ol{-1/2}\Bigr],\\
	e_5(\lambda)&=e^{a_3L\lambda^{1/2}}\Bigl[1+\Ol{-1/2}\Bigr]=
		e^{a_3Lu}e^{ia_3Lv}\Bigl[1+\Ol{-1/2}\Bigr],
\end{align*}
which shows that 
$e_1$, $e_3$ and $e_5^{-1}$ are bounded functions with
$e_5^{-1}\rightarrow 0$ exponentially as $\lambda\rightarrow\infty$,
whereas $e_1^{-1}$, $e_3^{-1}$ and $e_5$ are unbounded.
Assuming it is nonsingular, we use diagonal matrix
\begin{equation*}
	\tilde{E}:=\text{diag}(e_1,e_3,1)
\end{equation*}
to write
\begin{align}
	\det (R_1-R_2)&=0\nonumber\\
	\det [\tilde{E}(R_1-R_2)\tilde{E}]&=0\nonumber\\
	\label{eq:spectral-reduced2}
	\det R_3&=0,
\end{align}
so that unbounded terms $e_1^{-1}$ and $e_3^{-1}$ are removed
from the spectral equation.

To derive the reduced spectral equation \eqref{eq:spectral-reduced2},
we calculate the entries of matrix $R_3$ while removing all
exponentially decaying terms,
\begin{align*}
	R_3(1,1)&=-e_1^2-1+r_{11}\left(1+\hat{r}_{11}\lambda^{-1/2}
		+\tilde{r}_{11}\lambda^{-1}\right)+\Ol{-3/2},\\
	R_3(1,2)&=-r_{12}\lambda^{-2}\left(1-\hat{r}_{12}\lambda^{-1/2}
	\right)+\Ol{-3},\\
	R_3(1,3)&=\Ol{-3},\\
	R_3(2,1)&=-r_{21}\lambda^{3/2}\left(1-\hat{r}_{12}\lambda^{-1/2}
		\right)+\Ol{1/2},\\
	R_3(2,2)&=-ie_3^2-1-\frac{2a_4(1+ia_3)}{(1+i)a_3}e_3^2\lambda^{-1}
		+r_{22}\lambda^{-1/2}\left[1+(\hat{r}_{22}-d_1)\lambda^{-1/2}
		\right]+\Ol{-3/2},\\
	R_3(2,3)&=-\frac{2}{1-i}e_3\Bigl[1+\frac{(1-ia_3)a_4}{a_3}\lambda^{-1}
		\Bigr]+\Ol{-3/2},\\
	R_3(3,1)&=\Ol{1/2},\\
	R_3(3,2)&=-\frac{2}{1+i}e_3\Bigl[1-\frac{(1+ia_3)a_4}{a_3}\lambda^{-1}
		\Bigr]+\Ol{-3/2},\\
	R_3(3,3)&=i+\frac{2a_4(1-ia_3)}{(1-i)a_3}\lambda^{-1}+\Ol{-3/2}.
\end{align*}
Finally, expanding $\det R_3$ about the matrix's first row and
using functions
\begin{equation*}
	H_1(\lambda):=-e_1^2+r_{11}-1,\qquad
		H_2(\lambda):=-e_3^2-i,
\end{equation*}
we write equation \eqref{eq:spectral-reduced2} in the form
\begin{equation}\label{eq:spectral-reduced3}
	D_1-D_2=\Ol{-3/2},
\end{equation}
where 
\begin{align*}
	D_1&:=H_1H_2+\left[ir_{22}H_1+r_{11}\hat{r}_{11}H_2
		\right]\lambda^{-1/2}
		+\left[\tilde{D}_1H_1+ir_{11}\hat{r}_{11}r_{22}
		+r_{11}\tilde{r}_{11}H_2\right]\lambda^{-1},\\
	\tilde{D}_1&:=ir_{22}(\hat{r}_{22}-d_1)+i4a_4e_3^2
		-\frac{2a_4}{1-i}\frac{1-ia_3}{a_3}(ie_3^2+1)
		-i\frac{2a_4}{1+i}\frac{1+ia_3}{a_3}e_3^2,\\
	D_2&:=ir_{12}r_{21}\lambda^{-1/2}\Bigl(1-2\hat{r}_{12}
		\lambda^{-1/2}\Bigr).
\end{align*}

Close inspection of \eqref{eq:spectral-reduced3} reveals that the terms
of the two leading orders of magnitude ($1$, $\lambda^{-1/2}$)
contain no piezoelectric parameters, only mechanical and control
parameters.
The third order terms ($\lambda^{-1}$) are the first to depend on
piezoelectric parameters.
This proves Theorem \ref{thm:weak-perturbation}.

Next, keeping terms of only the first two orders of magnitude in
\eqref{eq:spectral-reduced3},
we derive the second order spectral equation
\begin{equation}\label{eq:spectral2-order2}
	g_1(\lambda)g_2(\lambda)+\bigl[ig_1(\lambda)h_1(\lambda)
	+r_{11}\hat{r}_{11}g_2(\lambda)-ir_{12}r_{21}\bigr]\lambda^{-1/2}
	=\Ol{-1},
\end{equation}
with functions
\begin{align}
	\label{eq:g1-func}
	g_1(\lambda)&:=-e^{i2c_1\lambda} + r_{11}-1,\\
	\label{eq:g2-func}
	g_2(\lambda)&:=-e^{i2c_3\lambda^{1/2}}-i,\\
	\label{eq:h1-func}
	h_1(\lambda)&:=r_{22}+2c_4e^{i2c_3\lambda^{1/2}}
\end{align}
and 
\begin{equation}\label{eq:c1234-const}
	c_j:=a_jL,\quad j=1,2,3,4.
\end{equation}

To first order, equation \eqref{eq:spectral2-order2}
reads
\begin{equation}\label{eq:spectral2-order1}
	g_1(\lambda)g_2(\lambda)=\Ol{-1/2}.
\end{equation}
Omitting the remainder term leads to solutions
\begin{equation}
	\label{eq:branch1-leading}
	g_1(\tilde{\lambda}_{1,n})=0\,\,\Rightarrow\,\,
	\tilde{\lambda}_{1,n}=\frac{n\pi}{c_1}
		-i\frac{1}{2c_1}\ln (r_{11}-1),\quad n\in\N,
\end{equation}
subject to $r_{11}-1>0$, and
\begin{equation}
	\label{eq:branch2-leading}
	g_2(\tilde{\lambda}_{2,n})=0\,\,\Rightarrow\,\,
	\tilde{\lambda}_{2,n}=(n-1/4)^2\frac{\pi^2}{c_3^2},
		\quad n\in\N,
\end{equation}
which we refer to respectively as the unperturbed branches 1 and 2.
Therefore, the leading order term of the asymptotic approximation
of the eigenvalues consists of the two disjoint subsets
\eqref{eq:branch1-leading} and \eqref{eq:branch2-leading}.
Moving on, we use the leading order solutions
\eqref{eq:branch1-leading} and \eqref{eq:branch2-leading}
to derive asymptotic approximations of the set of eigenvalues.

\subsection{Perturbed branch 1}\label{sec:branch1}

First, we quantify the contribution of the remainder term omitted from
\eqref{eq:spectral2-order1},
leading to a first order approximation of a subset of spectrum
$\spec{A}$.

We fix $n\in\N$ and define disk $B_{\epsilon}(\tilde{\lambda}_{1,n})$ of
radius $\epsilon>0$ centered at $\tilde{\lambda}_{1,n}$.
For $\lambda\in B_{\epsilon}(\tilde{\lambda}_{1,n})$
we cannot conclude that the values  $\abs{g_2(\lambda)}$
are bounded away from zero.
This is due to the fact that function $g_2$ depends on
$\lambda^{1/2}$, and points 
$\{\tilde{\lambda}_{1,n}^{1/2}\}_{n\in\N}$
are approaching the real axis as $n\rightarrow\infty$,
where all points 
$\{\tilde{\lambda}_{2,m}^{1/2}\}_{m\in\N}$---the zeros
of $g_2$---are located.
This can be seen by evaluating $g_2(\lambda)$ for 
$\lambda\in B_{\epsilon}(\tilde{\lambda}_{1,n})$,
leading to
\begin{equation*}
	\abs{g_2(\lambda)}^2=2\left[\sin\left(2c_3\sqrt{\tfrac{\pi}{c_1}}
		n^{1/2}\right)+1\right]+\O{n^{-1/2}},
\end{equation*}
which shows that $\abs{g_2(\lambda)}$ is oscillating with $n$ within
$[0,2+\xi]$, where $\xi\rightarrow 0$ as $n\rightarrow\infty$.

To resolve that, we fix a small positive constant $\delta$ and consider the
subset $\{\tilde{\lambda}_{1,n}\}_{n\in\N^*}$, $\N^*\subset\N$,
for which
$\underset{\lambda}{\inf}\abs{g_2(\lambda)}>\delta$ when
$\lambda\in B_{\epsilon}(\tilde{\lambda}_{1,n})$,
leading to a countably infinite subset of the unperturbed branch 1 points.
For this subset, we may rewrite \eqref{eq:spectral2-order1}
as
\begin{equation*}
	g_1(\lambda)=\Ol{-1/2}.
\end{equation*}
Now, we employ Rouch\'{e}'s theorem to show that analytic functions
$g_1$ and $g_1+\Ol{-1/2}$ have the same number of zeros in
$B_{\epsilon}(\tilde{\lambda}_{1,n})$
for an appropriately selected $\epsilon$.

To evaluate the functions on $\partial B_{\epsilon}(\tilde{\lambda}_{1,n})$,
let $\lambda=\tilde{\lambda}_{1,n}+\epsilon e^{i\theta}$,
$\theta\in [-\pi,\pi)$.
It follows that
\begin{align*}
	g_1(\lambda)&=-e^{i2c_1\tilde{\lambda}_{1,n}} e^{i2c_1\epsilon e^{i\theta}}
		+r_{11}-1\\
	&=(r_{11}-1)\bigl(1-e^{i2c_1\epsilon e^{i\theta}}\bigr)\\
	&=i2c_1e^{i\theta}(r_{11}-1)\epsilon\bigl[1+\O{\epsilon}\bigr]\quad
		\text{as $\epsilon\rightarrow 0$},
\end{align*}
or
\begin{equation*}
	\abs{g_1(\lambda)}>M_1\epsilon\quad
		\text{on $\partial B_{\epsilon}(\tilde{\lambda}_{1,n})$}
\end{equation*}
with constant $M_1>0$ and sufficiently small $\epsilon$.
Moreover, for sufficiently large $n$ there exists constant $M_2>0$
such that
\begin{equation*}
	\abs{\Ol{-1/2}}=\abs{\O{n^{-1/2}}}\leq M_2n^{-1/2}\quad
		\text{on $\partial B_{\epsilon}(\tilde{\lambda}_{1,n})$}.
\end{equation*}
Therefore, for $\epsilon _n=2M_2M_1^{-1}n^{-1/2}$
\begin{equation*}
	\abs{g_1(\lambda)}>2M_2n^{-1/2}>\abs{\Ol{-1/2}}\quad
		\text{on $\partial B_{\epsilon}(\tilde{\lambda}_{1,n})$}.
\end{equation*}
Applying Rouch\'{e}'s theorem yields that functions $g_1$ and
$g_1+\Ol{-1/2}$ have the same number of zeros in
$B_{\epsilon_n}(\tilde{\lambda}_{1,n})$,
with $\epsilon_n=\O{n^{-1/2}}$ as $n\rightarrow\infty$.
Namely, the first order asymptotic approximation of the considered subset
of eigenvalues reads
\begin{equation}\label{eq:branch1-order1}
	\lambda_{1,n}=\tilde{\lambda}_{1,n}\bigl[1+\O{n^{-3/2}}\bigr],
		\quad n\in\N^*,
\end{equation}
with $\N^*\subset\N$ the infinite subset constructed above.

Next, we calculate a formula for the term of order $n^{-3/2}$ and
establish the new remainder term.
To do so, we look for solutions of the second order spectral equation
\eqref{eq:spectral2-order2}
of the form
\begin{equation*}
	\lambda =\tilde{\lambda}_{1,n}(1+w_1).
\end{equation*}
It follows that
\begin{align*}
	g_1(\lambda)&=(r_{11}-1)\bigl(1-e^{i2c_1\tilde{\lambda}_{1,n}w_1}
		\bigr),\\
	g_2(\lambda)&=g_2(\tilde{\lambda}_{1,n})+\O{n^{-1}},\\
	h_1(\lambda)&=h_1(\tilde{\lambda}_{1,n})+\O{n^{-1}}.
\end{align*}
Substituting into \eqref{eq:spectral2-order2}
and keeping terms of the first two orders of magnitude (1, $n^{-1/2}$)
yields
\begin{equation}\label{eq:branch1-K1}
	e^{i2c_1\tilde{\lambda}_{1,n}w_1}=
		\frac{1+\left[ih_1(\tilde{\lambda}_{1,n})
		g_2^{-1}(\tilde{\lambda}_{1,n})
		+\frac{r_{11}\hat{r}_{11}}{r_{11}-1}
		-i\frac{r_{12}r_{21}}{(r_{11}-1)g_2(\tilde{\lambda}_{1,n})}
		\right]\tilde{\lambda}_{1,n}^{-1/2}}
		{1+ih_1(\tilde{\lambda}_{1,n})g_2^{-1}(\tilde{\lambda}_{1,n})
		\tilde{\lambda}_{1,n}^{-1/2}}
		=:K_{1,n},
\end{equation}
or equivalently
\begin{equation}\label{eq:branch1-w1}
	w_{1,n}=-i\frac{1}{2c_1}\tilde{\lambda}_{1,n}^{-1}\ln K_{1,n}
		=\O{n^{-3/2}}.
\end{equation}

To calculate the new remainder term, we consider
$\lambda\in B_{\epsilon}(\hat{\lambda}_{1,n})$ 
with radius $\epsilon>0$ and center
$\hat{\lambda}_{1,n}=\tilde{\lambda}_{1,n}(1+w_{1,n})$.
Using function
\begin{equation*}
	G_1(\lambda):=g_1(\lambda)g_2(\tilde{\lambda}_{1,n})+\left[
		ig_1(\lambda)h(\tilde{\lambda}_{1,n})
		+r_{11}\hat{r}_{11}g_2(\tilde{\lambda}_{1,n})-ir_{12}r_{21}\right]
		\tilde{\lambda}_{1,n}^{-1/2}
\end{equation*}
we rewrite the second order spectral equation
\eqref{eq:spectral2-order2}
as 
\begin{equation*}
	G_1(\lambda)=\Ol{-1},
\end{equation*}
with $G_1(\hat{\lambda}_{1,n})=0$ as established above.
Applying Rouch\'{e}'s theorem in a way analogous to that shown earlier,
we derive the second order asymptotic approximation
\begin{equation}\label{eq:branch1-order2}
	\lambda_{1,n}=\tilde{\lambda}_{1,n}\bigl[
		1+w_{1,n}+\O{n^{-2}}\bigr],\quad n\in\N^*,
\end{equation}
with $w_{1,n}$ of order $n^{-3/2}$ given by \eqref{eq:branch1-w1}.

\subsection{Perturbed branch 2}\label{sec:branch2}

The proof of the perturbed branch 2 approximation follows the same steps
as that of perturbed branch 1, so we present only the key points.
First, we pick $n\in\N$ and form disk
$B_{\epsilon}(\tilde{\lambda}_{2,n})$
of radius $\epsilon>0$ centered at $\tilde{\lambda}_{2,n}$.
To make sure that $\underset{\lambda}{\inf}\abs{g_1(\lambda)}>0$
when $\lambda\in B_{\epsilon}(\tilde{\lambda}_{2,n})$
it is sufficient to assume that
\begin{equation*}
	\epsilon<\abs{\frac{1}{2c_1}\ln(r_{11}-1)},
\end{equation*}
with the right hand side being the distance along the imaginary axis
between any two points in sets
$\{\tilde{\lambda}_{1,n}\}_{n\in\N}$ and
$\{\tilde{\lambda}_{2,m}\}_{m\in\N}$.
Application of Rouch\'{e}'s theorem then yields the first order
approximation
\begin{equation}\label{eq:branch2-order1}
	\hat{\lambda}_{2,n}=\tilde{\lambda}_{2,n}+\O{1}
		=\tilde{\lambda}_{2,n}\left[1+\O{n^{-2}}\right],
		\quad n\in\N.
\end{equation}

Next, we look for solutions of \eqref{eq:spectral2-order2} of the form
\begin{equation*}
	\lambda =\tilde{\lambda}_{2,n}(1+w_2).
\end{equation*}
Substitution into \eqref{eq:spectral2-order2}
leads to the following transcendental equation for $w_2$,
\begin{equation}\label{eq:branch2-w2}
\left[(ir_{22}+2c_4)\tilde{\lambda}_{2,n}^{-1/2}
		-c_3\tilde{\lambda}_{2,n}^{1/2}w_2\right]
		\left(-e^{i2c_1\tilde{\lambda}_{2,n}}
		e^{i2c_1\tilde{\lambda}_{2,n}w_2}+r_{11}-1\right)
		-ir_{12}r_{21}\tilde{\lambda}_{2,n}^{-1/2}=0.
\end{equation}
We consider \eqref{eq:branch2-w2}
to be the solution for $w_{2,n}$ in implicit form.

Finally, we form $B_{\epsilon}(\hat{\lambda}_{2,n})$
of radius $\epsilon>0$ centered at
$\hat{\lambda}_{2,n}=\tilde{\lambda}_{2,n}(1+w_{2,n})$.
Reapplying Rouch\'{e}'s theorem yields the second order asymptotic
approximation
\begin{equation}\label{eq:branch2-order2}
	\lambda_{2,n}=\tilde{\lambda}_{2,n}\left[1+w_{2,n}+\O{n^{-3}}\right],
		\quad n\in\N,
\end{equation}
with $w_{2,n}$ of order $n^{-2}$ given by \eqref{eq:branch2-w2}.

\section{Discussion}\label{sec:discussion}

In the present work we demonstrate that the
addition of piezoelectric energy harvesting can be viewed as a
weak perturbation of the underlying beam model, by showing that
no piezoelectric parameters appear in the first two orders
of magnitude of the asymptotic approximation of spectrum
$\sigma(A)$.

In addition, we show that the leading order term of the
asymptotic approximation of $\sigma(A)$
consists of a two-branch structure, which is the same structure with
that identified in \cite{Bal2004} where only the mechanical part
of the present model was considered.
Furthermore, we prove that the two-branch structure is
retained in the second order approximation of an infinite subset
$\sigma^*(A)\subset\sigma(A)$,
but not necessarily in that of the whole spectrum.

However, the present proof offers no information on the behavior
of complement $\sigma(A)\setminus\sigma^*(A)$, which may consist
of eigenvalues breaking the unperturbed structure.
Additionally, the remainder term derived in Section \ref{sec:branch2}
for the first order approximation of the perturbed branch 2 eigenvalues
is of order 1; namely, the remainder is not necessarily decreasing
with the eigenvalue number $n$ as $n\rightarrow\infty$.
These results show that additional work is required to fully understand
the role of the unperturbed structure in the higher order asymptotic
approximations of $\sigma(A)$.

To address these shortcomings, we plan to solve the second order
asymptotic equation \eqref{eq:spectral2-order2}
and the original spectral problem \eqref{eq:spectral-A}
numerically.
Doing so will allow us to verify the asymptotic analysis
and generate results for subset $\sigma(A)\setminus\sigma^*(A)$
not covered by the derived asymptotic approximation.
Our goal is to combine the asymptotic and numerical results to derive
a first order asymptotic approximation of the whole spectrum 
$\sigma(A)$.

Having an asymptotic approximation of the whole spectrum will enable us
to use existing operator technology to study the Riesz basis property
for the governing operator's set of eigenvectors
\cite{Bal2004,Shubov2014},
which remains an open problem for the coupled bending-torsion beam.

\appendix
\section{Omitted proofs}\label{sec:omitted-proofs}

\begin{proof}[\emph{\textbf{Proof of Lemma \ref{lem:inner-prod}}}]
It follows immediately that the functional is linear in the first argument,
conjugate symmetric, and that $f=0$ yields $\innprod{f}{f}=0$.
Next, given that
\begin{equation*}
	\abs{S\left(f_1\bar{f}_3+f_3\bar{f}_1\right)}\leq 2S\abs{f_1}\abs{f_3},
\end{equation*}
it is also true that
\begin{equation*}
\begin{split}
	C&\geq m\abs{f_1}^2 + J\abs{f_3}^2 - 2S\abs{f_1}\abs{f_3}\\
	&=\left(\sqrt{m}\abs{f_1} - \sqrt{J}\abs{f_3}\right)^2
		+ 2\sqrt{mJ}\abs{f_1}\abs{f_3} - 2S\abs{f_1}\abs{f_3}\\
	&=\left(\sqrt{m}\abs{f_1} - \sqrt{J}\abs{f_3}\right)^2
		+ 2\frac{D}{\sqrt{mJ}+S}\abs{f_1}\abs{f_3}\\
	&\geq 0
\end{split}
\end{equation*}
since $D>0$, which means that $\innprod{f}{f}\geq 0$
for any $f\in\widetilde{\mathcal{H}}$.
Finally, if $\innprod{f}{f}=0$ then nonnegativity implies that
\begin{equation*}
	f_0(x)=c_1x+c_2,\qquad f_2(x)=c_3,\qquad f_4=0,
\end{equation*}
for constants $c_1$, $c_2$ and $c_3$, as well as
\begin{equation*}
	\left(\sqrt{m}\abs{f_1}-\sqrt{J}\abs{f_3}\right)^2
		+2\frac{D}{\sqrt{mJ}+S}\abs{f_1}\abs{f_3}=0,
\end{equation*}
which yields that $f_1$, $f_3=0$.
Enforcing the boundary conditions encoded in
$\widetilde{\mathcal{H}}$ requires that $f_0$, $f_2=0$;
namely, $f=0$.
\end{proof}

\begin{proof}[\emph{\textbf{Proof of Lemma \ref{lem:norm-equiv}}}]
Let $f\in\widetilde{\mathcal{H}}$.
Using inequality
\begin{equation*}
	\abs{f_1\bar{f}_3+f_3\bar{f}_1}\leq 2\abs{f_1}\abs{f_3}
		\leq \abs{f_1}^2+\abs{f_3}^2,
\end{equation*}
we write
\begin{equation*}\begin{split}
	\norm{f}^2&\leq\frac{1}{2}\int_0^L\biggl[E\abs{f_0''}^2
		+m\abs{f_1}^2 + G\abs{f_2'}^2 + J\abs{f_3}^2
		+S\left(\abs{f_1}^2 + \abs{f_3}^2\right)\biggr]dx
		+\frac{1}{2}C_p\abs{f_4}^2\\
	&\leq\frac{1}{2}\int_0^L\biggl[E\sum_{j=0}^2\abs{f_0^{(j)}}^2
		+(m+S)\abs{f_1}^2 + G\sum_{k=0}^1\abs{f_2^{(k)}}^2
		+(J+S)\abs{f_3}^2\biggr]dx + \frac{1}{2}C_p\abs{f_4}^2\\
	&\leq C{\norm{f}}_1^2,
\end{split}\end{equation*}
with finite positive constant
$C:=\frac{1}{2}\max\left(E,\, m+S,\, G,\, J+S,\, C_p\right)$.

To prove the converse inequality, we begin by calculating finite positive
constants $c_0$ and $c_2$ such that
\begin{equation*}
	\norm{f_0''}_{L^2}^2\geq c_0\norm{f_0}_{H^2}^2,\qquad
	\norm{f_2'}_{L^2}^2\geq c_2\norm{f_2}_{H^1}^2.
\end{equation*}
Since $f_0\in C^{\infty}([0,L])$ with $f_0(0)=f_0'(0)=0$,
we may express $f_0$ in the form
\begin{equation*}
	f_0(x)=\int_0^x\int_0^yf_0''(z)dzdy,\quad
	\text{with $f_0''\in C([0,L])$}.
\end{equation*}
Then, using the Cauchy-Schwarz inequality we find that
\begin{equation*}
	\int_0^L\abs{f_0'(x)}^2dx\leq L^2\int_0^L\abs{f_0''(x)}^2dx
\end{equation*}
and
\begin{equation*}
	\int_0^L\abs{f_0(x)}^2dx\leq L^4\int_0^L\abs{f_0''(x)}^2dx,
\end{equation*}
which yield
\begin{equation*}
	\norm{f_0''}_{L^2}^2\geq c_0\norm{f_0}_{H^2}^2,\quad
	\text{with $c_0:=\left(L^4+L^2+1\right)^{-1}$}.
\end{equation*}
Similarly, for $f_2$ we find that
\begin{equation*}
	\norm{f_2'}_{L^2}^2\geq c_2\norm{f_2}_{H^1}^2,\quad
	\text{with $c_2:=\left(L^2+1\right)^{-1}$}.
\end{equation*}
We are now in a position to write
\begin{equation*}\begin{split}
\norm{f}^2&\geq\frac{1}{2}\int_0^L\biggl[E\abs{f_0''}^2
	+m\abs{f_1}^2 + G\abs{f_2'}^2 + J\abs{f_3}^2
	-S\left(\abs{f_1}^2+\abs{f_3}^2\right)\biggr]dx
	+\frac{1}{2}C_p\abs{f_4}^2\\
&\geq\frac{1}{2}\int_0^L\biggl[Ec_0\sum_{j=0}^2\abs{f_0^{(j)}}^2
	+(m-S)\abs{f_1}^2 +Gc_2\sum_{k=0}^1\abs{f_2^{(k)}}^2
	+(J-S)\abs{f_3}^2\biggr]dx +\frac{1}{2}C_p\abs{f_4}^2\\
&\geq c{\norm{f}}_1^2,
\end{split}\end{equation*}
with $c:=\frac{1}{2}\min\left(Ec_0,\, m-S,\, Gc_2,\, J-S,\, C_p\right)$
finite and positive if $S<m$ and $S<J$.
\end{proof}

\begin{proof}[\emph{\textbf{Proof of Lemma \ref{lem:domain-dense}}}]
For given $g\in\mathcal{H}$ and $\epsilon>0$,
we construct function $f_{\epsilon}\in\dom A$ such that
$\norm{f_{\epsilon}-g}_1<C\epsilon$ for a finite positive constant $C$.
Invoking Lemma \ref{lem:norm-equiv} then yields the desired result.

We begin by setting $f_4:=g_4\in\C$.
Next, the denseness of $C_c^{\infty}((0,L))$ in $L^2([0,L])$
provides the existence of $f_1\in C_c^{\infty}((0,L))$
such that $\norm{f_1-g_1}_{L^2}^2<\epsilon$
\cite{LiebLoss,HunterNacht}.
It follows that  $f_1\in H^2((0,L))$
and $f_1(0)=f_1'(0)=f_1'(L)=0$.
The same argument yields
$f_3\in C_c^{\infty}((0,L))\subset H^1((0,L))$
with $f_3(0)=f_3(L)=0$
such that $\norm{f_3-g_3}_{L^2}^2<\epsilon$.

Since $g_2\in H^1((0,L))$ with $g_2(0)=0$, there is
$g_2'\in L^2([0,L])$ such that
\begin{equation*}
	g_2(x)=\int_0^xg_2'(y)dy,\quad x\in[0,L].
\end{equation*}
Now, there exists $f_2'\in C_c^{\infty}((0,L))$
such that $\norm{f_2'-g_2'}_{L^2}^2<\epsilon$.
Using that, we form $f_2\in C^{\infty}([0,L])$ as
\begin{equation*}
	f_2(x):=\int_0^xf_2'(y)dy,\quad x\in [0,L],
\end{equation*}
which implies that $f_2\in H^2((0,L))$ with $f_2(0)=f_2'(L)=0$.
Using the Cauchy-Schwarz inequality, we write
\begin{equation*}
	\norm{f_2-g_2}_{L^2}^2\leq
		L^2\int_0^L\abs{f_2'(y)-g_2'(y)}^2dy
		=L^2\norm{f_2'-g_2'}_{L^2}^2
\end{equation*}
and
\begin{equation*}
\norm{f_2-g_2}_{H^1}^2=\norm{f_2-g_2}_{L^2}^2+
	\norm{f_2'-g_2'}_{L^2}^2<C_2\epsilon,
\end{equation*}
with $C_2:=1+L^2$.

Since $g_0\in H^2((0,L))$ with $g_0(0)=g_0'(0)=0$,
there is $g_0''\in L^2([0,L])$ for which
\begin{equation*}
	g_0(x)=\int_0^x\int_0^yg_0''(z)dzdy,\quad x\in[0,L].
\end{equation*}
We truncate and extend $g_0''$ to form $\tilde{f}_0''\in L^2(\R)$,
\begin{equation*}
	\tilde{f}_0''(x):=\begin{cases}
		0& x<0,\\
		g_0''(x)& x\in(0,L-\epsilon),\\
		\alpha& x\in(L-\epsilon,L+\epsilon),\\
		0& x>L+\epsilon,
		\end{cases}
\end{equation*}
with constant $\alpha:=-C_IE^{-1}g_4\in\C$.
Using positive constant $\delta<\epsilon$ and mollifier 
$\phi_{\delta}\in C_c^{\infty}(\R)$,
we form $f_0''\in C^{\infty}([0,L])$ as
\begin{equation*}
	f_0''(x):=\int_{\R}\phi_{\delta}(x-y)\tilde{f}_0''(y)dy,
		\quad x\in[0,L],
\end{equation*}
which ensures that $f_0''$ satisfies the desired conditions
$f_0''(L)=\alpha$ and $f_0'''(L)=0$.
Next, we define $f_0\in C^{\infty}([0,L])$ by
\begin{equation*}
	f_0(x):=\int_0^x\int_0^yf_0''(z)dzdy,\quad x\in[0,L],
\end{equation*}
so that $f_0\in H^4((0,L))$ with $f_0(0)=f_0'(0)=0$,
$f_0''(L)=\alpha$ and $f_0'''(L)=0$.
Denoting by $\tilde{f}_0''|_{[0,L]}$ the restriction of
$\tilde{f}_0''$ to $[0,L]$, we write
\begin{equation*}
	\norm{f_0''-g_0''}_{L^2}^2<
		\norm{f_0''-\tilde{f}_0''|_{[0,L]}}_{L^2}^2
		+\norm{\tilde{f}_0''|_{[0,L]}-g_0''}_{L^2}^2.
\end{equation*}
For any $u\in L^2(\R)$, its mollification $\phi_{\delta}\ast u$
converges to $u$ in $L^2(\R)$ as $\delta\downarrow 0$
\cite{LiebLoss,HunterNacht}.
It follows that the restriction of
$f_0''\equiv\phi_{\delta}\ast\tilde{f}_0''$ to $[0,L]$
converges in $L^2([0,L])$ to the corresponding restriction
of $\tilde{f}_0''$ as $\delta\downarrow 0$;
namely, that
\begin{equation*}
	\norm{f_0''-\tilde{f}_0''|_{[0,L]}}_{L^2}^2
		<\tilde{c}_0\delta<\tilde{c}_0\epsilon,
\end{equation*}
for a finite positive constant $\tilde{c}_0$.
Additionally,
\begin{equation*}
	\norm{\tilde{f}_0''|_{[0,L]}-g_0''}_{L^2}^2
		=\int_{L-\epsilon}^L\abs{\alpha-g_0''(x)}^2dx
		<(\alpha+\norm{g_0''}_{L^{\infty}})\epsilon.
\end{equation*}
Therefore,
\begin{equation*}
	\norm{f_0''-g_0''}_{L^2}^2<c_0\epsilon,
\end{equation*}
with $c_0:=\tilde{c}_0+\alpha+\norm{g_0''}_{L^{\infty}}<\infty$.
Finally, using the Cauchy-Schwarz inequality as earlier,
\begin{align*}
	\norm{f_0'-g_0'}_{L^2}^2&\leq L^2\norm{f_0''-g_0''}_{L^2}^2,\\
	\norm{f_0-g_0}_{L^2}^2&\leq L^4\norm{f_0''-g_0''}_{L^2}^2,
\end{align*}
namely,
\begin{equation*}
	\norm{f_0-g_0}_{H^2}^2<C_0\epsilon,
\end{equation*}
with constant $C_0:=(1+L^2+L^4)c_0$.

For $f_{\epsilon}:=(f_0,f_1,f_2,f_3,f_4)$,
the above demonstrates that
$f_{\epsilon}\in\dom A\subset\mathcal{H}$ and
\begin{equation*}
	\norm{f_{\epsilon}-g}_1^2<C\epsilon,
\end{equation*}
with positive constant $C:=C_0+C_2+2<\infty$.
\end{proof}

\begin{proof}[\emph{\textbf{Proof of Lemma \ref{lem:positive-imag}}}]
We denote by $\{\mu_j\}$ and $\{\lambda_j\}$, $j\in\N$, the
sets of eigenvalues of operators $iA$ and $A$ respectively.
Since $\lambda_j=-i\mu_j$ for all $j\in\N$, proving that
$\Re\mu_j<0$ for all $j\in\N$ implies the desired result.
To prove that, we consider the dynamical system formed by
\eqref{eq:op-evolution} for $f\in\dom A\subset\mathcal{H}$
and show that the norm induced by \eqref{eq:inner-prod}
defines a Liapunov function.
The result then follows from the Liapunov stability method
\cite{Perko2001}.

We use the problem's PDE formulation
\eqref{eq:harvester}--\eqref{eq:harvester-BCs}
and denote by $\mathcal{E}(t)$ the total energy as a function of time $t$,
which is equivalent to the norm induced by \eqref{eq:inner-prod}.
We then have that
\begin{align*}
	\mathcal{E}(t)&=\frac{1}{2}\int_0^L\left[Ew''\bar{w}''+m\dot{w}
		\bar{\dot{w}}+G\theta'\bar{\theta}'+J\dot{\theta}\bar{\dot{\theta}}
		+S\left(\dot{w}\bar{\dot{\theta}}+\dot{\theta}\bar{\dot{w}}\right)
		\right]dx+\frac{1}{2}C_pv\bar{v}\\
	\dot{\mathcal{E}}(t)&=\int_0^L\Re\bigl[Ew''\bar{\dot{w}}''
		+m\dot{w}\bar{\ddot{w}} +G\theta'\bar{\dot{\theta}}'
		+J\dot{\theta}\bar{\ddot{\theta}} +S(\dot{w}\bar{\ddot{\theta}}
		+\dot{\theta}\bar{\ddot{w}})\bigr]dx
		+C_p\Re\,(v\bar{\dot{v}}).
\end{align*}
From equations \eqref{eq:harvester} we derive respectively,
\begin{align*}
	\Re\,(m\dot{w}\bar{\ddot{w}}+S\dot{w}\bar{\ddot{\theta}})
		&=-\Re\,(Ew''''\bar{\dot{w}}),\\
	\Re\,(J\dot{\theta}\bar{\ddot{\theta}}+S\dot{\theta}\bar{\ddot{w}})
		&=\Re\,(G\theta''\bar{\dot{\theta}}),\\
	C_p\Re\,(v\bar{\dot{v}})&=-\frac{1}{R}\abs{v}^2
		-C_D\Re[\bar{\dot{w}}'(t,L)v].
\end{align*}
Using those and boundary conditions \eqref{eq:harvester-BCs} we find
\begin{equation*}
	\dot{\mathcal{E}}(t)=-k_1\abs{\dot{w}'(t,L)}^2
		-k_2\abs{\dot{\theta}(t,L)}^2-\frac{1}{R}\abs{v}^2
		-(C_I+C_D)\Re\,[\bar{\dot{w}}'(t,L)v].
\end{equation*}
If $C_I=-C_D$ then $\mathcal{E}(t)$ is monotone decreasing,
which yields the desired result.
\end{proof}

\section*{Acknowledgements}
I thank Marianna A. Shubov for her advice and support related to this work.
This work was funded in part by a Dissertation Year Fellowship awarded
by the Graduate School of the University of New Hampshire.

\bibliography{sources}

\begin{thebibliography}{10}

\bibitem{Abdelkefi2011}
A.~Abdelkefi, F.~Najar, A.~H. Nayfeh, and S.~B. Ayed.
\newblock An energy harvester using piezoelectric cantilever beams undergoing
  coupled bending-torsion vibrations.
\newblock {\em Smart Mater. Struct.}, 20(11):115007, 2011.

\bibitem{Abdelkefi2012}
A.~Abdelkefi, A.~H. Nayfeh, M.~R. Hajj, and F.~Najar.
\newblock Energy harvesting from a multifrequency response of a tuned
  bending-torsion system.
\newblock {\em Smart Mater. Struct.}, 21(7):075029, 2012.

\bibitem{AdamsSobolev}
R.~A. Adams and J.~J.~F. Fournier.
\newblock {\em Sobolev spaces}.
\newblock Academic Press, Boston, 2nd edition, 2003.

\bibitem{Bal2000}
A.~V. Balakrishnan.
\newblock Control of structures with self-straining actuators: coupled
  {Euler}/{Timoshenko} model.
\newblock In S.~Sivasundaram, editor, {\em Nonlinear problems in aviation and
  aerospace}, pages 179--194. Gordon \& Breach, Reading, 2000.

\bibitem{Bal2004}
A.~V. Balakrishnan, M.~A. Shubov, and C.~A. Peterson.
\newblock {Spectral analysis of coupled Euler-Bernoulli and Timoshenko beam
  model}.
\newblock {\em Z. Angew. Math. Mech.}, 84(5):291--313, 2004.

\bibitem{BirmanSolomiak}
M.~S. Birman and M.~Z. Solomiak.
\newblock {\em Spectral theory of self-adjoint operators in {Hilbert} space}.
\newblock D. Reidel, Dordrecht, 1987.

\bibitem{Bishop1989}
R.~E.~D. Bishop, S.~M. Cannon, and S.~Miao.
\newblock On coupled bending and torsional vibration of uniform beams.
\newblock {\em J. Sound Vib.}, 131(3):457--464, 1989.

\bibitem{Chen1987}
G.~Chen, S.~G. Krantz, D.~W. Ma, C.~E. Wayne, and H.~H. West.
\newblock The {Euler}-{Bernoulli} beam equation with boundary energy
  dissipation.
\newblock In S.~J. Lee, editor, {\em Operator methods for optimal control
  problems}. Marcel Dekker, New York, 1987.

\bibitem{Chen1990}
G.~Chen and J.~Zhou.
\newblock The wave propagation method for the analysis of boundary
  stabilization in vibrating structures.
\newblock {\em SIAM J. Appl. Math.}, 50(5):1254--1283, 1990.

\bibitem{Coleman1993}
M.~P. Coleman and H.~Wang.
\newblock Analysis of vibration spectrum of a {Timoshenko} beam with boundary
  damping by the wave method.
\newblock {\em Wave Motion}, 17(3):223--239, 1993.

\bibitem{Cook2008}
K.~A. Cook-Chennault, N.~Thambi, and A.~M. Sastry.
\newblock Powering {MEMS} portable devices--a review of non-regenerative and
  regenerative power supply systems with special emphasis on piezoelectric
  energy harvesting systems.
\newblock {\em Smart Mater. Struct.}, 17(4):043001, 2008.

\bibitem{Dokumaci1987}
E.~Dokumaci.
\newblock An exact solution for coupled bending and torsion vibrations of
  uniform beams having single cross-sectional symmetry.
\newblock {\em J. Sound Vib.}, 119(3):443--449, 1987.

\bibitem{Erturk2008}
A.~Erturk and D.~J. Inman.
\newblock A distributed parameter electromechanical model for cantilevered
  piezoelectric energy harvesters.
\newblock {\em J. Vib. Acoust.}, 130(4):041002, 2008.

\bibitem{Geist1997}
B.~Geist and J.~R. McLaughlin.
\newblock Double eigenvalues for the uniform {Timoshenko} beam.
\newblock {\em Appl. Math. Lett.}, 10(3):129--134, 1997.

\bibitem{Geist1998}
B.~Geist and J.~R. McLaughlin.
\newblock Eigenvalue formulas for the uniform {Timoshenko} beam: the free-free
  problem.
\newblock {\em Electron. Res. Ann. Amer. Math. Soc.}, 4(3):12--17, 1998.

\bibitem{Geist2001}
B.~Geist and J.~R. McLaughlin.
\newblock Asymptotic formulas for the eigenvalues of the {Timoshenko} beam.
\newblock {\em J. Math. Anal. Appl.}, 253(2):341--380, 2001.

\bibitem{HunterNacht}
J.~K. Hunter and B.~Nachtergaele.
\newblock {\em Applied analysis}.
\newblock World Scientific, New Jersey, 2001.

\bibitem{Inman2006}
D.~J. Inman and B.~L. Grisso.
\newblock Towards autonomous sensing.
\newblock In {\em {Proc. SPIE 6174 Smart Structures and Materials}}, San Diego,
  2006.

\bibitem{LiebLoss}
E.~H. Lieb and M.~Loss.
\newblock {\em Analysis}.
\newblock AMS, Providence, 2nd edition, 2001.

\bibitem{Liu2018}
H.~Liu, J.~Zhong, C.~Lee, S.-W. Lee, and L.~Lin.
\newblock A comprehensive review on piezoelectric energy harvesting technology:
  materials, mechanisms, and applications.
\newblock {\em Appl. Phys. Rev.}, 5(4):041306, 2018.

\bibitem{Perko2001}
L.~Perko.
\newblock {\em Differential equations and dynamical systems}.
\newblock Springer, New York, 2001.

\bibitem{BRao1998}
B.~Rao.
\newblock Optimal energy decay rate in a damped {Rayleigh} beam.
\newblock {\em DCDS}, 4(4):721--734, 1998.

\bibitem{Roundy2004}
S.~Roundy and P.~K. Wright.
\newblock A piezoelectric vibration based generator for wireless electronics.
\newblock {\em Smart Mater. Struct.}, 13(5):1131--1142, 2004.

\bibitem{Safaei2019}
M.~Safaei, H.~A. Sodano, and S.~R. Anton.
\newblock A review of energy harvesting using piezoelectric materials:
  state-of-the-art a decade later (2008–2018).
\newblock {\em Smart Mater. Struct.}, 28(11):113001, 2019.

\bibitem{Shubov1999}
M.~A. Shubov.
\newblock Spectral operators generated by {Timoshenko} beam model.
\newblock {\em Syst. Control Lett.}, 38(4--5):249--258, 1999.

\bibitem{Shubov2002}
M.~A. Shubov.
\newblock {Asymptotic and spectral analysis of the spatially nonhomogeneous
  Timoshenko beam model}.
\newblock {\em Math. Nachr.}, 241(1):125--162, 2002.

\bibitem{Shubov2014}
M.~A. Shubov.
\newblock On the completeness of root vectors generated by systems of coupled
  hyperbolic equations.
\newblock {\em Math. Nachr.}, 287(13):1497--1523, 2014.

\bibitem{Shubov2014b}
M.~A. Shubov.
\newblock {Spectral asymptotics, instability and Riesz basis property of root
  vectors for Rayleigh beam model with non-dissipative boundary conditions}.
\newblock {\em ASY}, 87(3-4):147--190, 2014.

\bibitem{Shubov2016}
M.~A. Shubov.
\newblock Asymptotic representation for the eigenvalues of a non-selfadjoint
  operator governing the dynamics of an energy harvesting model.
\newblock {\em Appl. Math. Optim.}, 73:545--569, 2016.

\bibitem{Shubov2017}
M.~A. Shubov.
\newblock Spectral analysis of a non-selfadjoint operator generated by an
  energy harvesting model and application to an exact controllability problem.
\newblock {\em ASY}, 102(3-4):119--156, 2017.

\bibitem{Shubov2018}
M.~A. Shubov.
\newblock Asymptotic and spectral analysis of a model of the piezoelectric
  energy harvester with the {Timoshenko} beam as a substructure.
\newblock {\em Appl. Sci.}, 8(9):1434, 2018.

\bibitem{Shubov2004}
M.~A. Shubov and C.~A. Peterson.
\newblock {Asymptotic analysis of nonselfadjoint operators generated by coupled
  Euler-Bernoulli and Timoshenko beam model}.
\newblock {\em Math. Nachr.}, 267(1):88--109, 2004.

\bibitem{ValesPhD}
C.~Vales.
\newblock {\em Asymptotic spectral analysis of a coupled bending-torsion beam
  energy harvester}.
\newblock {PhD dissertation}, University of New Hampshire, 2024.

\end{thebibliography}
\bibliographystyle{abbrv}

\end{document}